\documentclass[a4paper, 12pt]{amsart}
\usepackage{amsmath,latexsym,amssymb}
\usepackage{amscd}
\usepackage[
backend=bibtex,
bibencoding=utf8,
style=alphabetic,
sorting=nyt,
]{biblatex}
\bibliography{055books}
\usepackage{hyperref}

\newcommand{\ind}[1]{\underset{#1}{\mbox{\,\raisebox{-0.5ex}[0pt][0pt]{$\stackrel
{\text{\raisebox{-0.7ex}[0pt][0pt]{$|$}}}{\smile}$}\,}}}

\newlength {\xxxIndep}
\newlength {\yyyIndep}

\newcommand{\indep}{\mbox{\,\raisebox{-0.5ex}[0pt][0pt]
{$\stackrel{\text{\raisebox{-0.7ex}[0pt][0pt]{$|$}}}{\smile}$}\,}}
\def\Aut{{\rm Aut}}

\def\N{{\mathbb{N}}}

\def\acl{{\rm acl}}

\def\dcl{{\rm dcl}}

\def\tp{{\rm tp}}

\def\eq{{\rm eq}}

\def\aa{\overline{a}}

\def\N{\mathbb{N}}

\def\M{\mathcal{M}}
\def\C{\mathcal{C}}

\def\Psi{(P6) }

\def\C{\mathbb{C}}

\def\a{\hat{a}}

\def\P{\mathcal{P}}

\def\a{\hat{a}}

\usepackage{todonotes}

\newtheorem{theorem}{Theorem}[section]
\newtheorem{lemma}[theorem]{Lemma}
\newtheorem{proposition}[theorem]{Proposition}
\newtheorem{corollary}[theorem]{Corollary}

\newtheorem{definition}[theorem]{Definition}
\newtheorem{remarks}[theorem]{Remarks}
\newtheorem{remark}[theorem]{Remark}

\newtheorem{question}[theorem]{Question}

\newenvironment{proofclaim}{\noindent{}}{%
  \hspace*{\fill}\({\Box_{\text{Claim}}}\)\par\vskip.5ex}

\def\N{{\mathbb{N}}}

\def\acl{{\rm acl}}
\def\dcl{{\rm dcl}}
\def\tp{{\rm tp}}
\def\Aut{{\rm Aut}}

\def\N{\mathbb{N}}

\def\M{\mathcal{M}}
\def\C{\mathcal{C}}

\def\Psi{(P6) }

\def\C{\mathbb{C}}

\def\M{\mathcal{M}}

\def\a{\widehat{a}}

\def\P{\mathcal{P}}
\def\C{\mathcal{C}}
\def\MM{\mathbb{M}}
\def\aa{\bar{a}} 

\def\Ind#1#2{#1\setbox0=\hbox{\({#1x}\)}\kern\wd0\hbox to 0pt{\hss\({#1\mid}\)\hss}
\lower.9\ht0\hbox to 0pt{\hss\({#1\smile}\)\hss}\kern\wd0}
\def\Notind#1#2{#1\setbox0=\hbox{\({#1x}\)}\kern\wd0\hbox to 0pt{\mathchardef
\nn="3236\hss\({#1\nn}\)\kern1.4\wd0\hss}\hbox to 0pt{\hss\({#1\mid}\)\hss}\lower.9\ht0
\hbox to 0pt{\hss\({#1\smile}\)\hss}\kern\wd0}
\def\ind{\mathop{\mathpalette\Ind{}}}

\begin{document}
\section*{}

\title[Higher amalgamation properties]{Higher amalgamation properties in stable theories}

\author{David M. Evans}

\address{%
Department of Mathematics\\
Imperial College London\\
London\\
UK.}

\email{david.evans@imperial.ac.uk}

\author{Jonathan Kirby}
\address{
School of Mathematics\\
University of East Anglia\\
Norwich\\
UK.
}
\email{jonathan.kirby@uea.ac.uk}

\author{Tim Zander}
\address{
  Vision and Fusion Laboratory (IES)\\
  Karlsruhe Institute of Technology (KIT)\\
  Karlsruhe\\
  Germany.
}

\email{tim.zander@kit.edu}

\date{\today}

\begin{abstract} For a complete, stable theory $T$ we construct, in a reasonably canonical way, a related stable theory $T^*$ which has higher independent amalgamation properties over the algebraic closure of the empty-set. The theory $T^*$ is an algebraic cover of $T$ and we give an explicit description of the finite covers involved in the construction of $T^*$ from $T$. This follows an approach of E. Hrushovski.  If $T$ is almost strongly minimal with a $0$-definable strongly minimal set, then we show that $T^*$ has higher amalgamation over any algebraically closed subset.
\newline
\textit{2010 Mathematics Subject Classification:\/} 03C45; 03C99.
\end{abstract}
\maketitle

\section*{Introduction}

Given a (multisorted) complete, stable theory \({T}\) we wish to construct 
 in a reasonably canonical way a stable theory \({T^*}\) in which
 \({T}\) is fully embedded and which has \textit{higher amalgamation properties} (definitions can be found in section 1). This problem was considered in E. Hrushovski's paper \cite{hrushovski} where $T^*$ is obtained by adjoining certain \textit{finite covers}  to $T$ as extra sorts (see  4.3 and 4.11 of  \cite{hrushovski}). The extra sorts are referred to as \textit{generalised imaginary sorts} in \cite{hrushovski}.
 
 In this paper we provide a more detailed version of Hrushovski's construction of $T^*$: the main result is Theorem~\ref{uniquenesscor}. We  give two related approaches to the proof of this result. The first follows the
 approach of \cite{hrushovski} rather closely and we first describe
 this.  In \cite{hrushovski} the theory \({T^*}\) is obtained from
 \({T}\) by adjoining certain finite covers of \({T}\). In particular,
 this is done for 3-uniqueness (and \({4}\)-amalgamation) in 4.3 of
 \cite{hrushovski}, with a precise identification of the finite covers
 which need to be adjoined. The key point is that certain finite
 covers of \({T^*}\) with finite kernel should split over \({T^*}\)
 and by freely adjoining them as extra sorts to \({T}\) we guarantee
 this splitting. For higher amalgamation properties, a similar idea is
 sketched in (\cite{hrushovski}, 4.11): again the key point is that
 certain finite covers, but in general not having finite kernel,
 should split over \({T^*}\). In the general case, the finite covers which need to be adjoined here are not made explicit in \cite{hrushovski}, but are described in detail here.

The second approach to the proof, contained in Section 5, is more direct and avoids the group-theoretic notion of splitting. It makes use of use of a generalisation of the notion of a witness to the failure of amalgamation introduced in \cite{goodkol}. 

The theory $T^*$ which we construct has independent $n$-amalgamation over the algebraic closure of the empty-set, for all natural numbers $n$. In general, it does not then follow that $T^*$ has independent $n$-amalgamation over all algebraically closed sets. 
A natural example of this can be found in the theory $CCM$ of compact, complex manifolds. The algebraic closure of the empty set is a model of the theory and so $CCM$ has  $n$-amalgamation over $\acl(\emptyset)$ for all $n$. However, by Theorem 2.1 of \cite{MR3624418}, it fails to have $4$-amalgamation over some algebraically closed set. Further examples and an attempt towards classifying theories of this type can be found in the Third Author's PhD thesis \cite{uea63347}. 

In Section 6 we give a condition (`separable forking') on $T$ which guarantees that $T^*$ has higher amalgamation properties over all algebraically closed sets (Theorem~\ref{total}). In particular, this occurs if $T$ has Lascar rank 1 or is almost strongly minimal having a strongly minimal set defined over $\acl(\emptyset)$ (Lemma~\ref{sec:separ-indep-noti-4}). We pose the following:

\medskip

\noindent\textit{Question:\/} Suppose $T$ is a complete, stable theory. Is there a stable theory $T^{**}$ in which $T$ is fully embedded and which has $n$-amalgamation over all algebraically closed sets, for all $n$?

\medskip

It is shown in \cite{uea63347} that in the case where $T$ is $\omega$-categorical and $\omega$-stable, we may take $T^{**}$ to be an algebraic cover of $T$. 

\medskip

We give a brief over-view of the paper. Section 1 contains the definitions of the higher amalgamation properties (from \cite{hrushovski}) and some equivalent conditions to these. Of particular importance is the property $B(N)$ in Definition~\ref{BN} and we will work with this throughout when verifying the higher amalgamation properties. 

Section 2 introduces the key construction of the finite covers (`definable finite covers') which we adjoin to $T$ in order to produce $T^*$. This follows \cite{hrushovski} closely, though we provide more detailed proofs.

Section 3 discusses the notion of splitting of a finite cover and links it to the property $B(N)$. Section 4 provides some details about assembling the definable finite covers into the required $T^*$.

In Section 5, we provide an alternative perspective which avoids mentioning the group-theoretic notion of splitting. Finally in Section 6, we prove the result on theories of Lascar rank 1.

\medskip

\textit{Remarks and Acknowledgements:\/} Some of the material here appeared in some sparsely-circulated notes of the first Author in 2009. Since then, the theory of higher amalgamation functors has been extensively developed by Goodrick, Kim and Kolesnikov, and their approach and results have been used here. The other parts of the material come from the PhD thesis of the Third Author \cite{uea63347}. We thank E. Hrushovski for some useful comments on an earlier version of this and for the proof of Lemma~\ref{oldfinitecover}. The third author thanks UEA for supporting the PhD with a UEA studentship and the KASTEL project by the Federal Ministry of Education and Research(Germany), BMBF 16KIS0521, for funding.

\medskip

\noindent\textit{Notation:\/} Much of our notation (and abuse of notation) is standard. All theories should be assumed complete and first-order unless otherwise stated. We often work with multi-sorted theories $T$ and usually assume that imaginaries are included (so $T = T^{\eq}$). We denote by $\acl^M$ algebraic closure in the structure $M$, suppressing the $M$ where this is clear from the context. Similarly $\dcl^M$ denotes definable closure. `Definable' means `$0$-definable'.

We often work in a monster model $\mathbb{M}$ and elements or sets are then small subsets of this and models are elementary submodels. If $X, Y \subseteq \mathbb{M}$, then $\Aut(X/Y)$ denotes the group of permutations of $X$ which extend to automorphisms of $\mathbb{M}$ fixing all elements of $Y$. We sometimes (particularly in Section 3) consider this as a topological group with the topology of pointwise convergence.

\section{Higher amalgamation properties: definitions and preliminaries}

For this section, we suppose that \({T}\) is a complete, stable \({L}\)-theory and \({\mathbb{M}}\) is a monster model  of \({T}\). For our purposes, we may assume that \({T}\) has quantifier elimination. We will usually assume that \({T = T^{eq}}\): at any point, adding extra imaginary sorts will not cause problems.

\medskip

As \({T}\) is stable, any complete type \({p}\) over an algebraically closed set \({C}\)  (including imaginaries) has, in a definable way, a canonical extension \({p\vert D}\) to a type over any superset \({D \supseteq C}\), namely its unique non-forking extension. So 
for each \({L(C)}\)-formula \({\phi(x,y)}\) there is an \({L(C)}\)-formula \({\psi^p_\phi(y)}\) with the property that  
\[ p \vert D = \{ \phi(x,d) : \phi(x,y) \mbox{ an \({L(C)}\)-formula, }  d\in D \mbox{ and } \models \psi^p_\phi(d)\}\]
is a complete type over \({D}\), and \({p\vert C = p}\).

\medskip

The following definitions are taken from Hrushovski's paper \cite{hrushovski}, with some slight modifications from \cite{kimgoodkol}.

 Let \({\C}\) be the category whose objects are the algebraically closed substructures of \({\MM^{eq}}\) and whose morphisms are elementary maps between these.  Let \({N \in \N}\). Then \({\P(N)^{-}}\) is the power set \({\P(N)}\) of \({[N] = \{0,\ldots, N-1\}}\) without the whole set, thought of as a category where the morphisms $Mor(\P(N)^-)$ are the inclusion maps. Similarly, we also think of \({\P(N)}\) as a category. If \({n \leq N}\) then we denote by \({[N]^n}\) the set of subsets of \({[N]  = \{0,\ldots, N-1\}}\) of size \({n}\).
 
 Thus if \({A : \P(N)^{-} \to \C}\) is a functor, then for each \({s \in \P(N)^{-}}\) we have an algebraically closed subset \({A(s)}\) of \({\M}\). If \({s' \subseteq s}\) we have an elementary map \({A(s' \to s): A(s') \to A(s)}\) and we denote the image of this in \({A(s)}\) by \({A_s(s')}\).

\medskip

An (independent) \textit{\({N}\)-amalgamation problem} for \({T}\) is a functor
\[ A : \P(N)^{-} \to \C\]
where \({A(\emptyset) = \acl(\emptyset)}\) and for any \({s \in \P(N)^{-}}\) the set  \({\{A_s(\{i\}) : i\in s\}}\) is independent over \({\emptyset}\), and \({A(s) = \acl(A_s(i) : i \in s)}\).

\medskip

A \textit{solution} to this is an extension of \({A}\) to a functor 
\[ \bar{A}: \P(N) \to \C\]
on the full power set, satisfying the same conditions (so including the case \({s = [N]}\)).

\medskip

We say that \({T}\) has \textit{\({N}\)-existence}  if every such amalgamation problem has a solution. We say that \({T}\) has  \textit{\({N}\)-uniqueness}  if every such amalgamation problem \({A : \P(N)^- \to \C}\) has at most one solution, up to isomorphism. Explicitly, this means that if \({\bar{A}, \bar{A}' : \P(N) \to \C}\) are solutions to \({A}\), then there is an isomorphism \({\theta: \bar{A}([N]) \to \bar{A}'([N])}\) such that for all \({s \in \P(N)^{-}}\) we have \({\bar{A}'(s \to [N]) = \theta \circ \bar{A}(s\to [N])}\).

\medskip

If \({X \subseteq \MM}\) then we denote by \({T_X}\) the theory obtained by adding constants for the elements of \({X}\). We can then relativise all of the above definitions. So for example, an \({N}\)-amalgamation problem \({A}\) \textit{over \({X}\)} is  an \({N}\)-amalgamation problem for \({T_X}\). Note in particular that in this case we have \({A(\emptyset) = \acl(X)}\) (where algebraic closure is in the sense of \({T}\)). We say that \({T}\) has  \({N}\)-existence/ uniqueness \textit{over} \({X}\), to mean that \({T_X}\) has \({N}\)-existence / uniqueness.

\smallskip

It is perhaps worth noting some examples here. As long as we include imaginaries, then stable theories have \({2}\)-existence and uniqueness over algebraically closed sets: this is precisely the reason for the  introduction of imaginaries into stabiltiy theory. They also have  \({N}\)-existence and uniqueness over a model for all \({N}\). A vector space of infinite dimension over a finite field has \({N}\)-existence and uniqueness for all \({N}\). However, the corresponding projective space does not have \({3}\)-uniqueness (if the field has at least 3 elements). 

\medskip The following terminology is from (\cite{kimgoodkol},
Definition 3.1). In notation such as
\({a_0,\ldots, \hat{a}_i, \ldots, a_N}\), the hat denotes that the
term \({a_i}\) is omitted.

\begin{definition}\label{BN} \rm 
Suppose \({T}\) is stable and \({N \geq 2}\). We say that \({B(N)}\) holds (over \({\acl(\emptyset)}\))
 if whenever \({a_0,\ldots, a_{N-1}}\) are independent over \({\acl(\emptyset)}\), then
\begin{multline*}
\quad\Aut(\acl(a_0,\ldots, a_{N-2})/\bigcup_{i = 0}^{N-2} \acl(a_0\ldots \a_i \ldots a_{N-2} {a_{N-1}})) = \\
\Aut(\acl(a_0,\ldots, a_{N-2})/\bigcup_{i = 0}^{N-2} \acl(a_0\ldots \a_i \ldots a_{N-2})).
\end{multline*}
Equivalently, if \({c \in \acl(a_0,\ldots, a_{N-2})}\) is in the definable closure of
 \[{\bigcup_{i = 0}^{N-2} \acl(a_0\ldots \a_i \ldots a_{N-2} {a_{N-1}})},\]
 then it is in the definable closure of \({\bigcup_{i = 0}^{N-2} \acl(a_0\ldots \a_i \ldots a_{N-2})}\).
\end{definition}

Of course, \({B(2)}\) holds by stability of \({T}\). 
The following is similar to Proposition 3.5 of \cite{hrushovski},
 but works over a fixed set, so requires a different proof.

\begin{proposition} \label{12}
The complete stable \({L}\)-theory \({T}\) has \({N}\)-uniqueness over \({\acl(\emptyset)}\) iff
 \({B(k)}\) holds over \({\acl(\emptyset)}\) for all \({2 \leq k \leq N}\).
\end{proposition}

To prove this, the following notation and terminology from
\cite{kimgoodkol} will be useful.  If \({a_0,\ldots, a_{N-1}}\) are
given and \({s \subseteq \{0,\ldots, N-1\}}\) then
\[{\bar{a}_s = \acl(\{a_i : i \in s\})}.\]  By \({\Aut(\aa_s)}\) we
mean the set of elementary bijections from \({\aa_s}\) to itself.

\begin{definition} \rm Suppose \({2 \leq k \leq N}\). We say that \({T}\) has \textit{relative \({(k,N)}\)-uniqueness} (over \({\acl(\emptyset)}\)) if whenever \({\{a_i : i < N\}}\) are independent (over \({\emptyset}\)) and \({(\sigma_u : u \in [N]^{k-1})}\) are such that \({\sigma_u \in \Aut(\aa_u)}\) and \({\sigma \vert \aa_v = id}\) for all \({v \subset u}\), then \({\bigcup_{u \in [N]^{k-1}} \sigma_u}\) is elementary.
\end{definition}

Note that by stability, \({T}\) has relative \({(2,N)}\)-uniqueness (over \({\acl(\emptyset)}\)).

\begin{lemma} 
(\cite{kimgoodkol}, Lemma 4.4) 
Suppose \({N \geq k \geq 2}\) and  \({T}\) has property \({B(k)}\) (over \({\acl(\emptyset)}\)).  
Then \({T}\) has relative \({(k,N)}\)-uniqueness over \({\acl(\emptyset)}\).
\end{lemma}

\begin{corollary} 
  Suppose \({N \geq 2}\) and \({T}\) has property \({B(\ell)}\) (over
  \({\acl(\emptyset)}\)) for all \({2 \leq \ell \leq N}\).  Then
  \({T}\) has relative \({(k,N)}\)-uniqueness (over
  \({\acl(\emptyset)}\)) for all \({2 \leq k \leq N}\).
\end{corollary}

\begin{corollary}
\label{16}\label{stableuniqueness} 
Suppose \({T}\) has property \({B(\ell)}\) (over
\({\acl(\emptyset)}\)) for \[{2 \leq \ell \leq N}.\] Suppose
\({1 \leq r \leq N}\) and \({\{ a_i : i < N\}}\) are independent over
\({\emptyset}\).  Let \[{\{ \sigma_u : u \in [N]^r\}}\] be such that
\({\sigma_u \in \Aut(\aa_u)}\) and \({\sigma_u(x) = \sigma_v(x)}\)
whenever \({x \in \aa_u\cap \aa_v}\). Then
\({\bigcup\{\sigma_u : u \in [N]^r\}}\) is an elementary map.
\end{corollary}

\begin{proof} We prove this by induction on \({r}\), the case
  \({r=1}\) being straightforward. Consider the compatible system of
  elementary maps \[{\{\tau_v : v \in [N]^{r-1}\}}\] given by
  \({\tau_v = \sigma_u \vert \aa_v}\) whenever \({v \subset u}\). By
  inductive hypothesis, the union of these is an elementary map so
  extends to an automorphism \({\tau}\) (of \({\MM}\)). By the
  previous result, relative \({(r+1, N)}\)-uniqueness holds. We can
  apply this to \({\{\tau^{-1}\sigma_u : u \in [N]^r\}}\) to obtain an
  elementary map \({\rho}\) extending all
  \({\tau^{-1}\sigma_u}\). Then \({\tau\rho}\) is an elementary map
  extending all \({\sigma_u}\), as required.
\end{proof}

\begin{proof}(of Propsition \ref{12}) Suppose first that \({B(k)}\)
  holds (over \({\acl(\emptyset)}\)) for all \({k \leq N}\). Let
  \({A : \P(N)^- \to \C}\) be an \({N}\)-amalgamation problem (over
  \({\acl(\emptyset)}\)) and suppose \({A', A'' : \P(N) \to \C}\) are
  solutions. We can assume that \({A'(s \to s')}\) is inclusion and
  (by stationarity) that \[{A'(s) = A''(s) = \aa_s}\] for all
  \({s \subseteq s' \in \P(N)}\). Note that
  \({A''(s\to s') = A'(s \to s') = A(s \to s')}\) if
  \({s \subseteq s' \in \P^-(N)}\), so is inclusion. The only maps we
  need consider under \({A''}\) are therefore \({A''(u \to [N])}\) for
  \({u \in [N]^{N-1}}\) and there exist elementary permutations
  \({\sigma_u}\) on \({A''(u)}\) with
  \({A''(u\to [N])(a) = \sigma_u(a)}\) for all \({a \in
    \aa_u}\). These satisfy the condition of Corollary \ref{16}, so
  there is an automorphism which extends all \({\sigma_u}\) (with
  \({u \in [N]^{N-1}}\)). Let \({\theta}\) be the restriction of this
  to \({A'([N])}\). Then for all \({s \in \P(N)^-}\) we have
  \({A''(s \to [N]) = \theta \circ A'(s\to [N])}\), so \({A', A''}\)
  are isomorphic solutions to \({A}\).

  Conversely, suppose that \({N}\)-uniqueness holds over
  \({\acl(\emptyset)}\). We show that \({B(N)}\) holds over
  \({\acl(\emptyset)}\). Suppose \({a_0,\ldots, a_{N-1}}\) are
  independent over \({\acl(\emptyset}\)
  \[\sigma \in \Aut(\acl(a_0,\ldots, a_{N-2})/\bigcup_{i = 0}^{N-2}
    \acl(a_0\ldots \a_i \ldots a_{N-2}))\] Define functors
  \({A', A'' : \P(N) \to \C}\) by letting \({A'(s) = A''(s) = \aa_s}\)
  for \({s \subseteq [N]}\) and \({A'(s\to t)}\) the inclusion map for
  \({s \subseteq t \subseteq [N]}\). We also let \({A''(s\to t)}\) be
  the inclusion map unless \({s = \{0,\ldots, N-2\}}\) and
  \({t = [N]}\) and we let
  \({A''(\{0,\ldots, N-2\}\to [N])(a) = \sigma(a)}\) for
  \({a \in \aa_{\{0,\ldots, N-2\}}}\). One checks that \({A', A''}\)
  are indeed functors and that they have the same restriction \({A}\)
  to \({\P^-(N)}\). Thus they are solutions to the
  \({N}\)-amalgamation problem \({A}\). Therefore, there is an
  isomorphism \({\theta}\) of \({\aa_{[N]}}\) such that
  \({A''(s\to [N]) = \theta\circ A'(s\to [N])}\) for all
  \({s \subseteq [N]}\). In particular, \({\theta}\) restricts to
  \({\sigma}\) on \({\acl(a_0,\ldots, a_{N-2})}\) and is the identity
  on \[{\acl(a_0,\ldots, \a_i,\ldots, a_{N-2},a_{N-1})}\] for all
  \({i < N-1}\), as required. \end{proof}

\medskip

The proof of the following is essentially that of Lemma 4.1(2) in \cite{hrushovski}.

\begin{lemma}\label{ntonp1}
Suppose \({T}\) has  \({m}\)-uniqueness over \({\acl(\emptyset)}\) for all \({m \leq N}\). 
Then \({T}\) has \({n}\)-existence over \({\acl(\emptyset)}\) for all \({n \leq N+1}\). 
\end{lemma}

\begin{proof} It suffices to prove \({(N+1)}\)-existence. 

  Suppose \({A : \P(N+1)^- \to \C}\) is an independent
  \({(N+1)}\)-amal\-gama\-tion problem over \({\acl(\emptyset)}\). There
  is a resulting independent \({2}\)-amal\-gamation problem over
  \({A(1,\ldots, N-1)}\) given by
  \[{A(1,\ldots, N-1) \to A(0,\ldots, N-1)}\] and
  \({A(1,\ldots, N-1) \to A(1, \ldots, N)}\). By stability, this has a
  solution \({A(0,\ldots, N-1) \stackrel{\alpha}{\to} C}\) and
  \({A(1,\ldots, N) \stackrel{\beta}{\to} C}\). Let \({c_i \in C}\) be
  the image of \({A(i)}\) in \({C}\). Define \({B : \P(N+1) \to \C}\)
  by \({B(s) = \acl(c_i : i\in s)}\) (and \({B(f)}\) is inclusion for
  \({f \in Mor(\P(N+1))}\)).

  One checks that \({B}\) is an independent amalgamation over
  \({\acl(\emptyset)}\). It remains to show that there exist
  isomorphisms \({f_s : A(s) \to B(s)}\) compatible with the maps
  \({A(s') \to A(s)}\) and \({B(s') \to B(s)}\), for
  \[{s' \subseteq s \in \P(N+1)^-}.\]

  Unless \({0, N \in s}\), we can take \({f_s}\) to be the composition
  \[{A(s) \to A(1,\ldots, N) \stackrel{\alpha}{\to} C}\] or
  \({A(s) \to A(1,\ldots, N) \stackrel{\beta}{\to} C}\). If
  \({0, N \in s}\) then we define \({f_s}\) by induction on
  \({\vert s \vert}\). The point is that by
  \({\vert s \vert}\)-uniqueness, there is an isomorphism
  \({f_s : A(s) \to B(s)}\) compatible with the (already defined)
  \[{f_{s'} : A(s') \to B(s')}\] (for \({s' \subset
    s}\)).
\end{proof}

\section{Algebraic covers and definable finite covers}  \label{Tp}

We work with multisorted theories / structures and for the moment we do not assume stability. 

\subsection{Preliminaries}
\begin{definition}\rm 

  For complete theories \({T' \supseteq T}\) in many-sorted languages
  \({L' \supseteq L}\) we say that \({T'}\) is an \textit{
    algebraic cover} of \({T}\) if whenever \({M' \models T'}\), then
  the restriction \({M}\) of \({M'}\) to the sorts of \({L}\) is a
  model of \({T}\), \({M}\) is embedded in \({M'}\) (meaning the
  \({0}\)-definable subsets of \({M}\) are the same in the \({L}\) and
  \({L'}\) senses) and stably embedded in \({M'}\) (meaning: the
  parameter-definable subsets of \({M}\) are the same in the \({L}\)
  and \({L'}\) senses) and \({M'}\) is in the algebraic closure of
  finitely many sorts of \({M}\).  We refer to $M$ here as the $T$-part of $M'$ and say that $M'$ is an algebraic cover of $M$.

 We say that $T'$ as above is a \textit{finite cover}
  of \({T}\) if there is a sort \({M_1}\) such that \({M'}\) is in the
  definable closure of \({M \cup M_1}\) and there is a
  \({0}\)-definable function from \({M_1}\) to \({M}\) which is
  (boundedly) finite-to-one.  Any algebraic cover is
  interdefinable with a sequence of finite covers. 

\end{definition}

\begin{lemma}\label{weigeneralfinitecover} 
 Suppose $T'$ is an algebraic cover of $T$ and $M' \models T'$ has $T$-part $M$. Suppose $T$ has weak elimination of imaginaries. Then for every $e \in M'$ we have $e \in \acl(\acl(e)\cap M)$.
\end{lemma}

\begin{proof}
As $e$ is algebraic over $M$, there is some $L'(M)$-formula $\psi(x,a)$ with $\psi[M', a]$ finite and containing $e$. We can take $\psi(x,a)$ isolating $\tp_{M'}(e/M)$. The equivalence relation $(\forall x)(\psi(x,y) \leftrightarrow \psi(x,z))$ is an $L'$-definable equivalence relation on $M$ and so is defined by some $L$-formula $\theta(y,z)$. Let $f \in M^{\eq}$ denote the equivalence class containing $a$. Then $f \in \acl(e)$ and $e \in \acl(f)$. Moreover, by wei in $M$, we have $f \in \dcl(\acl(f) \cap M)$ (in $M$). Thus $e \in \acl(\acl(e)\cap M)$, as required.
\end{proof}

\begin{lemma}\label{addingaclemptytoalgebraiccover}
  Let \({{M'}}\) be an algebraic cover of $M$ and suppose that $\acl^{M^{\eq}}(\emptyset) = \dcl^{M^{\eq}}(\emptyset)$. Let $A = \acl^{M'}(\emptyset)$ and consider the expansion $M'_A$ of $M'$ by constants for elements of $A$. 
 Then $M'_A$  is an algebraic cover of \({M}\).
\end{lemma}

\begin{proof}
  We clearly have that \({{M'}_{A}}\) is contained in
  \({\acl(M)}\).  Hence we need to check that \({M}\) is embedded and stably
  embedded in \({{M'}_{A}}\).  It is clear that the
  parameter definable subsets of \({M}\) with parameters from \({{M'_A}}\)
 are also parameter definable in $M'$ and therefore in $M$.  So $M$ is stably embedded in $M'_A$. 
 
 Take some $\emptyset$-definable subset $S$ of \({M'_A}\) which is contained in $M$. So there is $c' \in A$ and an $L'$-formula $\phi(x,y)$ with $S = \phi[M',c']$. We may assume that $c'$ is the canonical parameter for $S$ in $M'^{\eq}$. Let $c$ be the canonical parameter for $S$ in $M^{\eq}$. Then $c, c'$ are interdefinable in $M'^{\eq}$ and so $c \in \acl^{M^\eq}(\emptyset)$. Thus $c \in \dcl^{M^\eq}(\emptyset)$. It follows that $S$ is $\emptyset$-definable in $M$.
\end{proof}

\subsection{The construction}

The following construction is taken from the proof of
(\cite{hrushovski}, 4.3).\label{constructioncover}

\medskip

We work in a monster model of a complete \({L}\)-theory \({T}\). For
our purposes we can assume that \({L}\) is relational and \({T}\) has
quantifier elimination. Suppose \({\theta(x,y,z)}\) is an
\({L}\)-formula with the property that \({\theta(a,b,z)}\) is
algebraic for all \({a,b}\). (Note that if \({\theta_0(x,y,z)}\) is
any \({L}\)-formula and \({a_0, b_0}\) are such that
\({\theta(a_0, b_0, z)}\) is algebraic, realized by \({c_0}\), then
there is an \({L}\)-formula \({\theta(x,y,z)}\) having this property,
and such that \({\models \theta(a_0, b_0, c_0)}\) and
\({\models \theta(x,y,z) \to \theta_0(x,y,z)}\).)

Suppose \({p(x)}\) is a complete type over the monster model which is
definable over \({\emptyset}\). So for each \({L}\)-formula
\({\phi(x,y)}\) there is an \({L}\)-formula \({\psi^p_\phi(y)}\) with
the property that
\[ p \vert D = \{ \phi(x,d) : \phi(x,y) \mbox{ an \({L}\)-formula, }
  d\in D \mbox{ and } \models \psi^p_\phi(d)\}.\] (Note that if
\({T}\) is stable and \({p}\) is a stationary type over
\({\emptyset}\) then \({p\vert D}\) is the restriction of the global
non-forking extension of \({p}\) to \({D}\), so this notation is
consistent with what was used previously.) Fix \({\theta}\) as
above. Let \({M \models T}\) and \({M^*}\) be a sufficiently saturated
elementary extension of \({M}\). Let \({a^*\in M^*}\) realize $p \vert M$ and let 
\[C = \Theta(M, a^*)= \{( b, c^*) : c^* \in M^*, b \in M \mbox{
    and }M^*\models \theta(a^*, b, c^*)\}\] Note that by the
  algebraicity, this does not depend on the choice of
\({M^*}\).  We make the disjoint union \({M \cup C\cup
\{a^*\}}\) into a structure \({M^+ = C(M,
a^*)}\) by giving it the induced structure from \({(M^*, a^*)}\).

More formally we let \({L^+ \supset L}\) be a language with a new sort
\({NC}\), a function symbol \({\pi}\) from \({NC}\) to some
\({L}\)-sorts, a new constant symbol \({*}\), and for each atomic
\({L}\)-formula \({R}\) a new relation symbol \({NR}\). To make
\({M^+}\) into an \({L^+}\)-structure we give \({M}\) its
\({L}\)-structure, take \({NC(M^+) = C}\), define
\({\pi((b,c^*)) = b}\), interpret the new constant symbol as
\({a^*}\), and for a new \({n}\)-ary relation symbol \({NR}\) and
\({e_1, \ldots , e_n \in M^+}\) we write:
\[ M^+ \models NR(e_1,\ldots, e_n) \Leftrightarrow 
M^* \models R(e_1,\ldots, e_n).\]

It is clear that if \({a^*, a^{**} \models p\vert M}\) then
\({C(M, a^*)}\) and \({C(M, a^{**})}\) are isomorphic over \({M}\)
(assume \({M^*}\) is sufficiently homogeneous, and use an automorphism
over \({M}\) which takes \({a^*}\) to \({a^{**}}\)). By construction,
the map \({\pi}\) is finite-to-one.


\medskip

The following lemma show that \({T^+= Th(M^+)}\) does not depend on
the choice of \({M}\) (at least, if \({M}\) is \({\omega}\)-saturated)
and types in \({M^+}\) can be understood in terms of types over
\({a^*}\) in \({M^*}\).  \def\tM{\tilde{M}}

\begin{lemma} \label{types} With this notation, suppose
  \({M, \tM \models T}\) are \({\omega}\)-saturated.
\begin{enumerate} 
\item If \({M \preceq \tM}\) and \({a^* \models p\vert \tM}\) then \({C(M, a^*) \preceq C(\tM, a^*)}\).
\item If \({d, e}\) are (tuples) in \({M^+ = C(M, a^*)}\) then: 
  \[ \tp^{M^+}(d) = \tp^{M^+}(e) \Leftrightarrow \tp^{M^*}(d/a^*) =
    \tp^{M^*}(e/a^*).\]
\end{enumerate}\end{lemma}

\begin{proof} (1) It is clear that \({C(M,a^*)}\) is a substructure of
  \({C(\tM, a^*)}\). It suffices to prove the statement in the case
  where \({\tM}\) is strongly \({\omega}\)-homogeneous, and we may
  assume that \({M^*}\) is a monster model (or at least,
  \({\vert \tM \vert^+}\) - saturated and strongly homogeneous). By a
  variation on the Tarski-Vaught Test (cf. \cite{MR1221741},
  Exercise 2.3.5), it is enough to show that if \({c}\) is a finite
  tuple of elements of \({C(M,a^*)}\) and \({e \in C(\tM, a^*)}\),
  then there is an automorphism \({\alpha}\) of \({C(\tM, a^*)}\) with
  \({\alpha(c) = c}\) and \[{\alpha(e) \in C(M,a^*)}.\] Furthermore,
  we can assume that \({e \in \tM}\), becuase if
  \({\alpha(\pi(e)) \in M}\), then \({\alpha(e) \in C(M,a^*)}\).

Suppose \({d \in M}\) is (a tuple)  such that the locus of \({c}\) over \({da^*}\) (in \({M^*}\)) is as small as possible (equal to \({n}\), say), witnessed by the formula \({\zeta(a^*, z, d)}\).

\smallskip

\noindent\textit{Claim 1:\/} We have  \({\tp(c/da^*) \vdash \tp(c/Ma^*)}\) (types in \({M^*}\)).
\begin{proofclaim}
The claim is that if \({c'}\) has the same type as \({c}\) over \({da^*}\)
 then it has the same type as \({c}\) over \({Ma^*}\). 
If this were not the case, there would be a tuple \({d' \in M}\)
 such that \({c, c'}\) have different types over \({dd'a^*}\). 
But then the locus of \({c}\) over \({dd'a^*}\) is smaller than its locus over \({da^*}\), 
 contradicting the choice of \({d}\). 
\end{proofclaim}

\smallskip

\noindent\textit{Claim 2:\/} We have \({\tp(c/da^*) \vdash \tp(c/\tM a^*)}\) (types in \({M^*}\)).

\begin{proofclaim}
  Suppose not. Then there is \({d_1 \in \tM}\) such that the locus of
  \({c}\) over \({d_1a^*}\) has size \({m < n}\).  Thus there is an
  \({L}\)-formula \({\eta(x,z,y)}\) such that
  \({\models \eta(a^*,c,d_1)}\) and the formulas
  \({(\exists^{=m}z)\eta(x,z,d_1)}\) and
  \[{(\forall z)(\eta(x,z,d_1) \to \zeta(x,z,d))}\] are in
  \({\tp(a^*/\tM)}\).  As this type is definable over \({\emptyset}\)
  and \({M \preceq \tM}\), there is some
  \({d' \models (\exists^{=m}z)\eta(a^*,z,d')\wedge(\forall
    z)(\eta(a^*,z,d') \to \zeta(a^*,z,d))}\) But this contradicts the
    minimality of \({n}\).
\end{proofclaim} 

\smallskip

We now return to the main thread of the proof. As \({M}\) is \({\omega}\)-saturated, there is \({e' \in M}\) with \({\tp(e'/d) = \tp(e/d)}\). By homogeneity of \({\tM}\), there is \({\beta \in \Aut(\tM/d)}\) with \({\beta(e) = e'}\). By definability of \({p}\) over \({\emptyset}\) (and strong homogeneity of \({M^*}\)), \({\beta}\) extends to some \({\gamma \in \Aut(M^*/a^*)}\). As \({\gamma}\) fixes \({d}\) and \({a^*}\), it is clear that \({c}\) and \({\gamma(c)}\) have the same type in \({M^*}\) over \({da^*}\). So by Claim 2, they have the same type over \({\tM a^*}\). By strong homogeneity of \({M^*}\), there is therefore some \({\delta \in \Aut(M^*/\tM a^*)}\) with \({\delta(\gamma(c)) = c}\). Let \({\alpha = \delta\gamma}\). Then \({\alpha \in \Aut(M^*/a^*)}\), \({\alpha\tM = \tM}\), \({\alpha(cd) = cd}\) and \({\alpha(e) = e' \in M}\). Therefore \({\alpha}\) induces an automorphism of \({C(\tM, a^*)}\) which fixes \({cd}\) and sends \({e}\) to an element of \({M}\), as required.

\medskip

(2) The direction \({\Rightarrow}\) is clear. So suppose \({d, e}\) are tuples in \({M^+}\) with the same type over \({a^*}\) in \({M^*}\). We may assume (by (1)) that \({M^*}\) is a monster model and so there is an automorphism \({\alpha}\) of \({M^*}\) which fixes \({a^*}\) and takes \({d}\) to \({e}\). Furthermore, we can assume that there is an \({\omega}\)-saturated \({M'' \preceq M^*}\) with \({M \preceq M''}\),  \({\alpha(M'') = M''}\) and \({a^* \models p \vert M''}\). Thus \({\alpha}\) induces an automorphism of \({M^{++} = C(M'', a^*)}\) which takes \({d}\) to \({e}\), and therefore
\[ \tp^{M^{++}}(d) = \tp^{M^{++}}(e).\]
By (1) we have \({M^+ \preceq M^{++}}\) and so
\[\tp^{M^+}(d) = \tp^{M^{++}}(d) = \tp^{M^{++}}(e) = \tp^{M^{+}}(e).\qedhere\]
\end{proof}

\medskip

\begin{lemma}\label{oldfinitecover} \label{finitecoverold}
With the above notation, \({T^+}\) is a finite cover of \({T}\). 
\end{lemma}

\begin{proof} (Hrushovski) It will suffice to prove that if \({M^+ = M \cup C \cup \{a^*\}}\) is a saturated model of \({T^+}\) (with \({C = NC(M^+)}\) and  \({ M \subseteq M^+}\)  its \({L}\)-part), then \begin{enumerate}
\item[(1)] \({M}\) is a saturated model of \({T}\);
\item[(2)] any \({L}\)-automorphism of \({M}\) extends to an \({L^+}\)-automorphism of \({M^+}\). 
\end{enumerate}

Indeed, if \({M}\) is not embedded in \({M^+}\), then by saturation of \({M^+}\) there exist \({b, b' \in M}\) with the same \({L}\)-type and different \({L^+}\)-types, and this contradicts (1) and (2). As there is a finite-to-one map from \({M^+}\) to \({M}\), stable embeddedness follows from embeddedness. 

Note that \({M}\) is a model of \({T}\) and a relativised reduct of \({M^+}\), so (1) is automatic. To prove (2), we `embed' \({M^+}\) into a model of \({T}\). Of course, to make sense of this, we need to change the language. We consider \({M^+}\) as an \({L}\)-structure \({\tM^+}\) by interpreting each atomic relation \({R}\) of \({L}\) as the corresponding new relation \({NR}\). The definition of \({T^+}\) shows that the quantifier free diagram of this is consistent with \({T}\), so \({\tM^+}\) can be considered  as a substructure of a model \({M^*}\) of \({T}\), which we may take to be \({\vert M^+\vert ^+}\)-strongly homogeneous. It is clear that \({M \subseteq \tM^+}\) is a model of \({T}\) and \({a^* \models p\vert M}\) (in \({M^*}\)). Moreover, by the algebraicity, \({NC(M^+) = \Theta(M,  a^*)}\). By definability of \({p\vert M}\), any automorphism of \({M}\) extends to an automorphism of \({M^*}\) which fixes \({a^*}\). This stabilises the set \({M^+}\) and preserves the \({L^+}\)-structure on it, so gives an automorphism of \({M^+}\). This proves (2). \end{proof}

\begin{definition}\label{sec:algebr-covers-defin} \rm
  If \({T}\), \({p}\) and \({\theta}\)  are as above, we denote by
  \({T_{p,\theta}}\)  the \({L^+}\)-theory of
  \({C(M, a^*)}\), where \({M}\) is an \({\omega}\)-saturated model of
  \({T}\) and \({a^*\models p \vert M}\). We refer to this as a
  \textit{definable finite cover} of \({T}\). (We usually ignore the
  adjoined \({a^*}\).)

 \begin{remarks}\label{gendefcov}\rm  We note the following slight extension of this construction. Suppose $T, p, M, M^*$ are as above,  and for $i \leq r$ we have an $L$-formula $\theta_i(x,y,z_i)$ such that $\theta_i(a,b,z_i)$ is algebraic for all $a,b$. We can then let $M^+$ be the induced structure in $(M^*, a^*)$ on:
 \[C(M,a^*) = \bigcup_{i \leq r} \{(b,c^*) : c^* \in M^*, b \in M \mbox{ and } M^*\models \theta_i(a^*, b, c^*)\}.\]
 Exactly as before, this gives a finite cover of $M$. We also refer to this as a definable finite cover of $T$ and denote it by $T_{p,(\theta_i: i\leq m)}$.
 \end{remarks}

\end{definition}

\section{Splitting of finite covers}\label{sec:splitt-finite-covers}

We say that an algebraic cover \({T'}\) of \({T}\) \textit{splits} over \({T}\) if there is an expansion of \({T'}\)  to an algebraic cover \({T''}\) of \({T}\)  which is interdefinable with \({T}\). 

Splitting of \({T'}\) over \({T}\) implies that for any model \({M'}\) of \({T'}\) there is an expansion \({M''}\) of \({M'}\) with 
\[ \Aut(M') = \Aut(M'/M) \rtimes \Aut(M'').\]
Here \({\rtimes}\) denotes semi-direct product. Note that algebraicity means that with the usual automorphism group topology, the \textit{kernel} of the cover, \({\Aut(M'/M)}\),  is a profinite group. As a notational convenience, we shall freely confuse theories with their saturated models.

The following gives the consequence of splitting which we shall need. Algebraic closure is always in the \({eq}\)-sense.

\begin{lemma} \label{splitting} Suppose \({M \subseteq M'}\) is a
  split algebraic cover, \[{X_1, \ldots, X_r \subseteq M}\] and
  \({\acl^M(X_i) = X_i}\) for \({i = 1,\ldots, r}\). Then
\[\Aut(\bigcup_i \acl^{M'}(X_i) / \bigcup_i X_i) = \Aut(\bigcup_i \acl^{M'}(X_i) / M).\]
\end{lemma}

\begin{proof} The containment \({\supseteq}\) is clear. For the converse, it is enough to prove the case \({r = 1}\). So suppose \({h \in \Aut(M' / X_1)}\).  By assumption, there is an expansion \({M''}\) of \({M'}\) such that 
  \[\Aut(M') = \Aut(M'/M) \rtimes \Aut(M'').\]Thus we can write
  (uniquely) \({h = k.g}\) where \({k \in \Aut(M'/M)}\) and
  \[{g \in \Aut(M'')}.\] So clearly \({g \in \Aut(M''/X_1)}.\)

  Consider the restriction map from \({\Aut(M''/ X_1)}\) to
  \[{\Aut(\acl^{M'}(X_1)/ X_1)}.\] We claim that this has trivial
  image. To see this, note that this is a continous map and the range
  is a profinite group. Also, restriction to the sorts of \({M}\)
  gives an isomorphism from \({\Aut(M'')}\) to \({\Aut(M)}\). As
  \({X_1}\) is algebraically closed in \({M}\), it follows that
  \({\Aut(M/ X_1)}\), and therefore \({\Aut(M''/ X_1)}\), has no
  proper open subgroup of finite index. Thus the only possible
  continuous image of \({\Aut(M''/ X_1)}\) inside a profinite group is
  the trivial group. So \({g}\) fixes each element of
  \({\acl^{M'}(X_1)}\) and therefore \({h}\) and \({k}\) agree on
  \({\acl^{M'}(X_1)}\). As \({k \in \Aut(M'/ M)}\) this proves the
  lemma. \end{proof}

%

\medskip

For the rest of the section, \({T}\) is a complete \({L}\)-theory (with \({T=T^{\eq}}\)). We work in a monster model \({\mathbb{M}}\) of \({T}\) and other models are elementary substructures of this. The following is the main point. 

\begin{lemma} \label{usesplitting} Suppose that every definable finite
  cover of \({T}\) splits over \({T}\). Let \({p}\) be a global type
  definable over \({\emptyset}\), let \({M \models T}\) be
  \({\omega}\)-saturated and \({a^* \models p \vert M}\). Suppose
  \({b_0, b_1,\ldots, b_r \in M}\) and \({c \in \acl(a^*b_0)}\),
  \({e_i \in \acl(a^*b_i)}\) are such that $c \in \dcl(a^*e_1\ldots e_r M)$. Then 
  \[{c \in \dcl(a^*e_1\ldots e_r B_0B_1\ldots B_r)}\] where
  \({B_i = \acl(b_i)}\).
\end{lemma}

\begin{proof} For \({i = 1,\ldots, r}\), let \({\theta_i(x,y,z)}\)
  witness the algebraicity of \({e_i}\) over \({a^*, b_i}\) and let
  \({\theta_0(x,y,z)}\) witness algebraicity of \({c}\) over
  \({a^*b_0}\). Let
  \[{\theta(x,y,z) = \bigvee_{i= 0}^r \theta_i(x,y,z)}\] (there is a
  slight complication here in that we ought to assume that elements to
  be represented by the same variable come from the same sort, but for
  our purposes this can be arranged). Performing the construction of
  the previous section, we obtain a definable finite cover
  \({\pi : M^+ \to M}\) of \({M}\) with
  \({(b_0, c), (b_1,e_1), \ldots, (b_r, e_r) \in M^+}\). By Lemma
  \ref{types}
  \[(b_0, c) \in \dcl^{M^+}((b_1,e_1), \ldots, (b_r, e_r) M),\]
  therefore by splitting and Lemma \ref{splitting},
  \[(b_0, c) \in \dcl^{M^+}((b_1,e_1), \ldots, (b_r, e_r) B_0B_1\ldots
    B_r). \] A further application of Lemma \ref{types} then gives
  that \[c \in \dcl(a^*e_1\ldots e_r B_0B_1\ldots B_r),\] as required.
\end{proof}


\begin{proposition} \label{mainpoint} Suppose \({n}\) is an integer
  and \({n \geq 3}\). Suppose \[{p_1,\ldots, p_n}\] are global types
  which are definable over \({\emptyset}\) and \({a_1,\ldots, a_n}\)
  are such that
  \({ a_i \models p_i \vert a_1 \ldots \hat{a}_i \ldots a_n}\). Let
  \[ c\in \acl(a_1\ldots a_{n-1}) \cap
    \dcl(\bigcup_{i=1}^{n-1}\acl(a_1\ldots \hat{a}_i\ldots a_n)).\]
  Suppose further that every definable finite cover of \({T}\) splits
  over \({T}\).

Then

\[ c \in \dcl(\bigcup_{i=1}^{n-1}\acl(a_1\ldots \hat{a}_i\ldots a_{n-1})).\]
\end{proposition}

\begin{proof} Let \({M}\) be an \({\omega}\)-saturated model of \({T}\) which contains the \({a_i}\). For \({0 \leq i \leq n-1}\) we prove by induction on \({i}\) that
\begin{equation}\label{indhyp} c \in \dcl(\bigcup_{j \leq i} \acl(a_1\ldots \hat{a}_j\ldots a_{n-1}) \cup \bigcup_{i <j \leq n -1}\acl(a_1\ldots \hat{a}_j\ldots a_n))\end{equation}
with the interpretation that the case \({i = 0}\) is what we are given, and the case \({i = n-1}\) is what is required. So suppose (\ref{indhyp}) holds for some \({i < n-1}\) and deduce it for \({i+1}\).

Let  \({a^* \models p\vert M}\). As \({\tp(a_1\ldots a_{i+1}\ldots a_n) = \tp(a_1\ldots a^*\ldots a_n)}\) we can 
find \({c^*}\) with \({\tp(c^*a_1\ldots a^* \ldots a_{n}) = \tp(c a_1 \ldots a_{i+1} \ldots a_n)}\). By assumption (\ref{indhyp}) 
\({c^*}\) is in the definable closure of
\begin{multline*}\bigcup_{j\leq i}\acl(a^*a_1\ldots \hat{a}_j \ldots a_i a_{i+2}\ldots a_{n-1}) \cup \acl(a_1\ldots a_i a_{i+2}\ldots a_n)\\ \cup  \bigcup_{i+2 \leq j < n} \acl(a^*a_1\ldots a_i a_{i+2}\ldots \hat{a}_j \ldots a_n).\end{multline*}

In particular \({c^*}\) is in the definable closure of

\begin{multline*}\bigcup_{j\leq i}\acl(a^*a_1\ldots \hat{a}_j \ldots a_i a_{i+2}\ldots a_{n-1}) \cup M\\ \cup  \bigcup_{i+2 \leq j < n} \acl(a^*a_1\ldots a_i a_{i+2}\ldots \hat{a}_j \ldots a_n).\end{multline*}

Note that \({c^* \in \acl(a_1\ldots a_i a^*a_{i+2}\ldots a_{n-1})}\). So by Lemma \ref{usesplitting} it follows that  \({c^*}\) is in the definable closure  of

\begin{multline*}\bigcup_{j\leq i}\acl(a^*a_1\ldots \hat{a}_j \ldots a_i a_{i+2}\ldots a_{n-1}) \cup \acl(a_1\ldots a_ia^* a_{i+2}\ldots a_{n-1})\\ \cup  \bigcup_{i+2 \leq j < n} \acl(a^*a_1\ldots a_i a_{i+2}\ldots \hat{a}_j \ldots a_n).\end{multline*}

As \({\tp(c^*a_1\ldots a^* \ldots a_{n}) = \tp(c a_1 \ldots a_{i+1} \ldots a_n)}\), it follows that \({c}\) is in the definable closure of 

\begin{multline*}\bigcup_{j\leq i}\acl(a_1\ldots \hat{a}_j \ldots a_i a_{i+1} a_{i+2}\ldots a_{n-1}) \cup \acl(a_1\ldots a_i a_{i+2}\ldots a_{n-1})\\ \cup  \bigcup_{i+2 \leq j < n} \acl(a_1\ldots a_i a_{i+1}a_{i+2}\ldots \hat{a}_j \ldots a_n)\end{multline*}
which is the required inductive step. \end{proof}

\begin{remarks}\rm For the above proof to work, we need only assume that each definable cover of the form \({T_{p_i,  \theta}}\) splits over \({T}\) (for \({i = 1,\ldots n}\)). Also note that in the case \({n = 3}\) the proof of (\cite{hrushovski}, Proposition 3.3) shows that the definable finite covers \({T_{p_i, \theta}}\) arising in the proof have finite relative automorphism group. Thus for the case \({n=3}\), the hypothesis can be further weakened to assuming that each definable cover of the form \({T_{p_i, \theta}}\) and which has finite relative automorphism group splits over \({T}\).\end{remarks}

\begin{theorem} \label{35} Suppose \({T}\) is stable and
  \({\acl(\emptyset) = \dcl(\emptyset)}\) in \({T}\). Suppose every
  definable finite cover of \({T}\) splits over \({T}\). Then \({T}\)
  has \({N}\)-existence and \({N}\)-uniqueness for all \({N}\).
\end{theorem}

\begin{proof} Suppose \({n \geq 3}\) and \({a_1,\ldots, a_{n}}\) are
  independent over \({\emptyset}\). By the assumption
  \({\acl(\emptyset) = \dcl(\emptyset)}\), each type
  \({\tp(a_i/\emptyset)}\) is stationary, so has a unique global
  non-forking extension \({p_i}\). By independence,
  \[{a_i \models p_i\vert a_1 \ldots \hat{a}_i \ldots a_n}.\] So the
  condition required in property \({B(n)}\) follows from Proposition
  \ref{mainpoint}. Thus property \({B(n)}\) holds for all \({n}\) and
  therefore the result follows from Proposition \ref{12} and Lemma
  \ref{ntonp1}.
\end{proof}

\section{Amalgamating covers}\label{AmalgSec}

Suppose \({T}\) is a complete \({L}\)-theory and \({S}\) is a subset
of the sorts. We denote by \({L^S}\) the restriction of \({L}\) to the
sorts \({S}\) and \({T^S}\) the restriction of \({T}\). If
\({M \models T}\) we denote by \({M^S}\) the model of \({T^S}\) which
is the restriction of \({M}\) to the sorts in \({S}\).

\begin{lemma} Suppose that \({T^S}\) is fully embedded in \({T}\) and
  \({p^S(y)}\) is a global type of \({T^S}\) definable over
  \({\emptyset}\). Then there is a unique global type \({p(y)}\) of
  \({T}\) whose restriction to the \({S}\)-sorts is
  \({p^S}\). Moreover, \({p}\) is definable over \({\emptyset}\).
\end{lemma}

\begin{proof} For the first part, it suffices to show that if
  \({M \models T}\) and \[{a, a' \models p^S\vert M^S},\] then
  \({\tp(a/M) = \tp(a'/M)}\). Suppose not. Then there are
  \({c \in M}\) and an \({L}\)-formula \({\psi(x,y)}\) with
  \({\models \psi(c,a)\wedge \neg\psi(c,a')}\). By stable
  embeddedness, there is \({d \in M^S}\) and an \({L^S}\)-formula
  \({\phi(z,y)}\) such that
  \[{\models (\forall y)(\psi(c,y)\leftrightarrow \phi(d,y)).}\] But
  then \({\models \phi(d,a)\wedge \neg\phi(d,a')}\) contradicting
  \({a,a' \models p^S\vert M^S}\).

  To see that \({p(y)}\) is definable over \({\emptyset}\), let
  \({\psi(x,y)}\) be an \({L}\)-formula as above and consider
  \({C = \{c: \psi(c,y) \in p(y)\}}\). We can write (letting
  \({\phi(z,y)}\) range over \({L^S}\)-formulas):
  \[ C = \bigcup_\phi \{ c: \exists d \, (\phi(d,y) \in p^S \mbox{ and
    } \models \psi(c,y)\leftrightarrow \phi(d,y))\},\]
  \[ \mathbb{M} \setminus C = \bigcup_\phi \{ c : \exists d \,
    (\phi(d,y) \not\in p^S \mbox{ and } \models
    \psi(c,y)\leftrightarrow \phi(d,y))\}.\] These are unions of
  \({\emptyset}\)-definable sets, so \({C}\) is
  \({\emptyset}\)-definable.

\end{proof}

In the construction of Section \ref{Tp}, suppose that \({p^S}\) is a
type from \({S}\) definable over \({\emptyset}\). We wish to compare
\({(T^S)_{p^S,\theta}}\) and \({T_{p,\theta}}\), where \({p}\) is as
in the above Lemma. Before stating the result, we need some
generalities about amalgamating theories.
 
Suppose \({T', T''}\) are complete theories in which the collection of
sorts of \({T}\) is fully embedded (that is, embedded and stably
embedded). For our purposes we can assume that the languages
\({L', L''}\) of these are relational and have intersection the
language \({L}\) of \({T}\).  If \({M', M''}\) are models of
\({T', T''}\) with a common \({T}\)-part \({M\models T}\) we denote by
\({M' \coprod_M M''}\) the \({L'\cup L''}\)-structure which is the
disjoint union over \({M}\) of \({M'}\) and \({M''}\) and in which no
new instances of atomic relations are added. So any
\({\emptyset}\)-definable relation is a boolean combination of
products of \({\emptyset}\)-definable relations on \({M'}\) and
\({M''}\). The theory of this does not depend on the choice of
\({M', M''}\) and we denote it by \({T'\times_T T''}\). Clearly
\({\Aut(M' \coprod_M M'')}\) is the fibre product
\begin{multline*}
\Aut(M') \times_{\Aut(M)} \Aut(M'') =\\ \{(g',g'') \in \Aut(M')\times
    \Aut(M'') : g'\vert M = g''\vert M\}.
\end{multline*}
Note that \({T}\) is
fully embedded in \({T'\times_T T''}\) and this is an algebraic
extension of \({T}\) if \({T'}\) and \({T''}\) are.
 
\begin{lemma} \label{fiddly} With the above notation,
  \({T_{p,\theta}}\) and \({(T^S)_{p^S,\theta} \times_{T^S} T}\) are
  interdefinable.
\end{lemma}
\begin{proof}Suppose \({M \preceq M^*}\) are sufficiently saturated
  and strongly homogeneous, and \({a \in M^*}\) realises
  \({p \vert M}\). In the notation of Section 3, let
  \({M^+ = M_{p,\theta} = C(M,a)}\) and
  \[{(M^S)^+ = (M^S)_{p^S,\theta} = C(M^S, a)}.\] So of course
  \({M^+ = M \cup (M^S)^+}\), as sets. It will suffice to show that if
  \({g \in \Aut(M)}\) and \({h \in \Aut((M^S)^+)}\) have the same
  restriction to \({M^S}\), then \({g \cup h}\) extends to an
  automorphism of \({M^*}\) which fixes \({a}\). Note that \({g}\)
  extends to an automorphism \({g'}\) of \({M^*}\) which fixes \({a}\)
  (by the homogeneity and the definability of \({p \vert M}\)) so we
  may adjust by \({(g')^{-1}}\) and assume that \({g}\) is the
  identity on \({M}\) and therefore \({h \in \Aut((M^S)^+/M^S)}\).

  We now observe that if \({e, e' \in (M^*)^S}\) have the same type in
  \({(M^*)^S}\) over \({M^S}\) then they have the same type in
  \({M^*}\) over \({M}\). Indeed, if \({\phi(c,y)}\) is any formula
  with parameters \({c \in M}\) where \({y}\) is of sort in \({S}\),
  then by stable embeddedness of \({S}\) there exists \({d \in M^S}\)
  and \({\psi(d,y)}\) such that
  \[{M \models (\forall y)(\phi(c,y)\leftrightarrow \psi(d,y))}.\] As
  \({M \preceq M^*}\) we have
  \({M^* \models (\forall y)(\phi(c,y)\leftrightarrow
    \psi(d,y))}\). So as \({e, e'}\) have the same type over \({d}\),
  it follows that \({\models \phi(c,e) \leftrightarrow \phi(c,e')}\),
  as required.

  So if \({m}\) denotes some enumeration of \({(M^S)^+}\),
  then \[{\tp(m/M, a) = \tp(hm/M,a)},\] therefore \({h}\) extends to an
  automorphism of \({M^*}\) fixing \({M}\) and \({a}\), by the
  homogeneity. \end{proof}

\begin{definition}\label{strongfc}\rm We say that a
 finite cover \({T' \supseteq T}\) is a \textit{strong finite cover} if
  there is a finite subset \({S}\) of the sorts of \({T}\) such that
  \({T'}\) is interdefinable (up to imaginaries) with
  \({T''\times_{T^S} T}\) for some finite cover
  \({T'' \supseteq T^S}\) of \({T^S}\).  An algebraic cover is a
  \textit{strong algebraic cover} if it is interdefinable with a sequence of
  strong finite covers of the base.
\end{definition}

So the above lemma shows that the definable finite covers
\({T_{p,\theta}}\) are strong finite covers. The following terminology
will also be convenient.

\begin{definition} \rm A complete theory \({T}\) is said to be
  \textit{irreducible} if \[{\acl(\emptyset) = \dcl(\emptyset)}\] (in
  \({T^{\eq}}\)). We refer to the expansion of a theory \({T}\) by
  constants for \({\acl(\emptyset)}\) as its \textit{irreducible
    component}.
\end{definition}

Note that if \({M}\) is a sufficiently saturated model of \({T}\) and
\({M^\circ}\) its irreducible component, then
\({\Aut(M^\circ) = \Aut(M/\acl(\emptyset))}\) is the intersection of
the open subgroups of finite index in \({\Aut(M)}\) and this has no
proper open subgroup of finite index. Also note that if
\({T', T''}\) are irreducible and have a common fully embedded part
\({T}\), then \({T'\times_T T''}\) is irreducible.

\begin{lemma} Suppose \({T}\) is irreducible and \({T \subseteq T'}\)
  is a finite cover whose irreducible component splits over
  \({T}\). Then \({T'}\) splits over \({T}\).
\end{lemma}

\begin{proof} Let \({M'}\) be a saturated model of \({T'}\) and
  \({M}\) its \({T}\)-part. By assumption, there is a closed subgroup
  \({H}\) with
  \[{ \Aut(M' / \acl(\emptyset)) = \Aut(M'/M, \acl(\emptyset)) \rtimes
    H.}\] As \({M}\) is irreducible, restriction of \({H}\) to \({M}\)
  gives an isomorphism with \({\Aut(M)}\).  We want to show that
  \({ \Aut(M') = \Aut(M'/M) \rtimes H.}\) As
  \[{H \cap \Aut(M'/M, \acl(\emptyset)) = 1}\] and
  \({H \leq \Aut(M' / \acl(\emptyset))}\), it follows that
  \({H \cap \Aut(M'/M) = 1}\). To see that
  \({\Aut(M') = \Aut(M'/M)H}\), as
  \({\Aut(M'/M)H \geq \Aut(M'/\acl(\emptyset))}\), it will suffice to
  show that every element \({k}\) of \({\Aut(\acl(\emptyset))}\)
  extends to an element of \({\Aut(M'/M)}\). Extend \({k}\) to some
  automorphism \({h}\) of \({M'}\). There is \({g \in H}\) which
  agrees with \({h}\) on \({M}\), so \({g^{-1}h \in \Aut(M'/M)}\) and
  agrees with \({k}\) on \({\acl(\emptyset)}\).
\end{proof}

\begin{proposition}\label{46} Suppose \({T_0}\) is an irreducible,
  complete theory. Then there is an irreducible, strong algebraic
  extension \({T\supseteq T_0}\) with the property that every strong
  finite cover of \({T}\) splits over \({T}\).
\end{proposition}

\begin{proof} By the previous lemma, it is enough to ensure that every
  irreducible, strong finite cover of \({T}\) splits over
  \({T}\). Using amalgamation, we construct \({T}\) as the union of an
  \({\omega}\)-chain of irreducible, strong algebraic extensions
  \({T_0 \subseteq T_1 \subseteq T_2 \subseteq \ldots}\) with the
  property that \({T_{i+1}}\) contains a copy of every strong finite
  cover of \({T_i}\).

  More precisely, using such a chain, we can ensure that, for a saturated model \({M}\) of
  \({T}\), if \({S}\) consists of the sorts of \({T_0}\) together with
  a finite number of extra sorts and if \({M' \supseteq M^S}\) is a
  (strong) irreducible finite cover of \({M^S}\), then there is a
  (fully embedded) collection of sorts \({R \supseteq S}\) (with
  \({R \setminus S}\) finite) and a bijection
  \({\alpha : M' \to M^R}\) which is the identity on \({M^S}\) and
  such that \({\alpha, \alpha^{-1}}\) map \({\emptyset}\)-definable
  sets to \({\emptyset}\)-definable sets.

  Note that a strong finite cover of \({M}\) is of the form
  \({M'' = M' \coprod_{M^S} M}\) for some such \({M^S \subseteq
    M'}\). We require a splitting of the map
  \[{\rho : \Aut(M'') \to \Aut(M)}\] given by restriction to
  \({M}\). Let \({\alpha : M' \to M^R}\) be as in the previous
  paragraph. We define a map \({\gamma : \Aut(M) \to \Aut(M'')}\) as
  follows. If \({g \in \Aut(M)}\) let \({\gamma(g)(x)}\) equal
  \({\alpha^{-1}g\alpha(x)}\) if \({x \in M'}\) and equal \({g(x)}\)
  if \({x \in M}\). This is a continuous group embedding with closed
  image and \({\rho(\gamma(g)) = g}\) for all \({g \in \Aut(M)}\).
\end{proof}

If \({T_0}\) is not irreducible, then we can obtain a similar result,
but without the irreducibility in the conclusion. Combining
Proposition \ref{46} with Lemma \ref{fiddly} we obtain:

\begin{theorem} \label{47}
  Suppose \({T_0}\) is a complete, stable theory with the property that 
  \({\acl(\emptyset) = \dcl(\emptyset)}\) in \({T_0^{\eq}}\).  Then
  there is a strong algebraic extension \({T \supseteq T_0}\) with
  \({\acl(\emptyset) = \dcl(\emptyset)}\) in \({T^{eq}}\) and such
  that every definable cover \({T_{p,\theta}}\) of \({T}\) splits over
  \({T}\).
\end{theorem}

Putting this together with Theorem \ref{35} we obtain:

\begin{theorem}\label{uniquenesscor} 
  Suppose \({T_0}\) is a complete, stable theory  with the property that
  \({\acl(\emptyset) = \dcl(\emptyset)}\) in \({T_0^{\eq}}\).  Then there is
  an algebraic extension \({T \supseteq T_0}\) which has existence and
  uniqueness for independent \({n}\)-amalgamation over
  \({\acl(\emptyset)}\) for all \({n}\).
\end{theorem}

\begin{remark}\rm  Note that when we use the construction in Proposition \ref{46} to prove Theorems \ref{47} and \label{uniquenesscor} we can be more economical.  It suffices to construct $T_{i+1}$ (in the proof of \ref{46}) from $T_{i}$ by adjoining all \textit{definable} finite covers of $T_i$, rather than adjoining all strong finite covers.
\end{remark}

\section{An alternative approach: \({n}\)-witnesses}
\label{cha:witnesses-failure}

In this Section we will give an alternative proof of Theorem
\ref{uniquenesscor} which avoids the notion of splitting.

\subsection{Localise failure of \texorpdfstring{\(B(n)\)}{Bn}}

The next definition is only some re-writing of Property \({B(n)}\) in
a way which will be more suitable for our later needs. This was inspired by Definition 3.12 of \cite{Goodrick2015-GOOTPA-3}.
We use the following notation. By 
  \({\phi(\overline x,A)}\) we mean a formula
  \({\phi(\overline x,\overline a)}\) for some tuple \({\overline a}\)
  of \({A}\).

\begin{definition}\rm 
  We fix some stable theory \({T=T^{\eq}}\) and some \({n\ge 3}\).  Let the
  tuple \({a_1,\ldots ,a_n,f_1,\ldots ,f_n}\) consist of elements of the
  monster model of \({T}\) and $A$ be a subset of the monster model. We say that this tuple
  \({a_1,\ldots ,a_n,f_1,\ldots ,f_n}\) is an \index{witness!\({n}\)-}
  \emph{\index{\({n}\)-!witness}\({n}\)-witness} over \({A}\), if we
  have that it satisfies the following four properties:

  \begin{enumerate}
  \item \({a_1,\ldots ,a_n}\) are independent over \(A\),
  \item
    \(f_n \in \acl (a_1\ldots a_{n-1}A) - \dcl \bigl(
    \bigcup_{i=1}^{n-1}\acl(a_1\ldots \hat a_i\ldots
    a_{n-1}A)\bigr)\),
  \item for $i \leq n-1$ we have $f_i \in \acl(a_1\ldots \hat{a_i}\ldots a_nA)$. More explicitly, there exist formulae
    \[{\phi_{i}(x_1,\ldots \hat x_{i} \ldots ,x_{n},y;z_{i})}\] and natural numbers \({m_{i}}\) such that 
  \[\models \phi_{i}(a_1,\ldots \hat a_i\ldots ,a_{n},A;f_i)\] and
 \[{\models \forall x_{1},\ldots\hat x_{i}\ldots ,x_{n},y
     \exists^{<m_{i}}z\phi_{i}(x_1,\ldots \hat x_{i}\ldots
     ,x_{n-1},y;z_{i})},\]
\item \({f_n\in \dcl(f_1\ldots f_{n-1})}\).
\end{enumerate}
\end{definition}

\begin{lemma}\label{sec:localise-failure-bn}
  Property \(B(n)\) over \(A\) fails if and only if there exists an
  \(n\)-witness over \(A\).
\end{lemma}

\begin{proof}
  Assume that \(B(n)\) over \(A\) fails.  So this means that there
  exists some  \(a_{1},\ldots ,a_{n}\) independent over $A$ such that
  \begin{multline*}
    B= \Bigl(\acl (a_1 \ldots  a_{n-1}A)\cap
    \dcl\bigl(\bigcup_{i=1}^{n-1}\acl(a_1 \ldots \hat a_i\ldots  a_n  A)\bigr)\Bigr)\\
    - \dcl \bigl(\bigcup_{i=1}^{n-1}\acl(a_1 \ldots\hat a_i\ldots
     a_{n-1} A)\bigr)
  \end{multline*}
  is non-empty.  Then by definition of \(\dcl\) and \(\acl\) there
  exists \[f_{i}\in \acl(a_1 \ldots \hat a_i\ldots  a_n  A)\] such that
  \({ a_1,\ldots ,a_n,f_1,\ldots ,f_n}\) is an \({n}\)-witness over
  \({A}\). Pick any \(f_{n}\) in \(B\) and then find the
  \(f_{i}\in \acl(a_1 \ldots \hat a_i\ldots a_n A)\) accordingly such
  that \(f_n\in \dcl(f_1 \ldots  f_{n-1})\).  Then of course we can find
  formulae \(\phi_{i}\) such that the condition 3 is satisfied.  Of
  course if there is an \(n\)-witness, then \(B(n)\) fails as
  well, as we can easily reverse the above process.
\end{proof}

\subsection{The finite cover eliminates a witness}

\begin{proposition}\label{coveroverparameter}
  Let \({T(=T^\eq)}\) be a stable theory with $\acl(\emptyset) = \dcl(\emptyset)$.  Let  \({a_1,\ldots ,a_n}\) be  an
  independent sequence (over \(\emptyset\)) and let \({d}\) be any element of
  the set
  \begin{multline*}
    \Bigl(\acl(a_1\ldots a_{n-1})\cap \dcl\bigl( \bigcup_{j=1}^{n-1}
    \acl(a_1\ldots \hat a_j \ldots a_n )\bigr)\Bigr)\\
    -\dcl \bigl(\bigcup_{j=1}^{n-1}\acl(a_1\ldots \hat a_j \ldots
    a_{n-1})\bigr).
  \end{multline*}
  Then there exists a finite cover \({{T}^+}\) of \({T}\) such
  that
  \[d \in \dcl^+ \bigl(\bigcup_{j=1}^{n-1}
    \acl^{+}(a_1\ldots \hat a_j \ldots a_{n-1})\bigr),\] where by
\({\dcl^{+}}\) and \({\acl^{+}}\) we mean the
evaluation of \({\acl}\) and \({\dcl}\) in the theory
$T^+$.
\end{proposition}

\begin{proof}
   By
  Lemma~\ref{sec:localise-failure-bn} we can fix an \({n}\)-witness
  \({a_1,\ldots ,a_n}\), \({f_1,\ldots ,f_n}\) over \({\emptyset}\)
  with \({d=f_n}\).  Let
  \({\phi_{i}(x_1,\ldots\hat x_{i} \ldots,x_{n-1};z_{i})}\) for all
  \({i}\) with \({1\le i\le n-1}\) be the formulae satisfying
  condition \({3}\) of the witness definition.  Now extend \({T}\) to
  the finite cover \(T^+ = {T_{p,(\phi_{i}:1\le i\le n-1)}}\) 
  defined in Remark~\ref{gendefcov} with \({p=\tp(a_n)}\), and  $M$ in the construction is  an $\omega$-saturated model of $T$ containing $a_1,\ldots, a_{n-1}$.We may
  assume that
  \({a_n=a^* \models p\vert M}\),
  where \({a^*}\) is the new generic constant of the finite cover. Let $M^*$, $M^+$  be the models of $T$, $T^+$ in the construction. 

  Then because
  \(M^*\models {\phi_{i}(a_{1},\ldots \hat a_{i}\ldots ,a_{n};f_{i})}\)
  and \({\phi_{i}(a_{1},\ldots \hat a_{i}\ldots ,a_{n};z_{i})}\) is
  algebraic
  we have that
  \[{f_{i}\in \acl^{M^*}(a_1,\ldots \hat a_i\ldots ,a_{n-1})}.\] Now
  there is a formula \[{\psi(y_1,\ldots ,y_{n-1};y_n)}\] such that
  \({M^*\models\psi(f_1,\ldots ,f_{n-1};f_n)}\) and
  \[{M^*\models\exists^{=1}y\psi(f_1,\ldots ,f_{n-1};y)}.\] As in the construction of $M^+$, let 
  \(NR_{\psi}\) denotes the new relation
    symbol (coming from \({\psi}\)).  Hence we have that
  \[NR_{\psi}(f_1,\ldots ,f_{n-1};y)\] isolates
  \({f_{n}}\) and therefore shows that
  \({f_{n}\in \dcl^{M^+}(f_{j}:1\le j\le n-1)}\).  This shows that
  \[d \in \dcl^{M^+} \bigl(\bigcup_{i=1}^{n-1}
      \acl^{M^+}(a_1\ldots\hat a_i\ldots a_{n-1})\bigr).\] Hence
  \({T_{p,(\theta_{i}:1\le i\le m)}}\)  is the finite cover we were
  looking for.
\end{proof}

\subsection{Eliminate all witnesses}

Now we are able to prove \ref{uniquenesscor} by adding all definable finite cover as constructed in section \ref{Tp}
and then repeating this process \({\omega}\)-many times. The construction is essentially that of Proposition~\ref{46}.

\begin{proof}(of \ref{uniquenesscor})
  First a brief description of the proof. Add any possible finite
  cover constructed in section \ref{Tp}  to our
  theory.  Then use the Proposition~\ref{coveroverparameter} to note
  that \({B(n)}\) over \({\emptyset}\) is true for any independent
  sequence of old sort.  Then repeat this process \({\omega}\)-many
  times to eliminate any malicious behaviour (in terms of failure of
  \({B(n)}\)) in any of these new algebraic covers.

  We start the real proof. 
  Denote the language of \({T_0}\) by
  \({L_0}\).  Fix a saturated model \({M}\) of \({T_0}\).  We
  will construct a chain \({(M_i:i\in\omega)}\) of algebraic covers of
  \({M}\) with language \({L_i}\) such that
  \({M_i={(M_i^\eq)}_{\acl^\eq(\emptyset)}}\), \({M_0=M}\) and \({M_i}\)
  is  saturated.   Note that it will be enough to prove that
  \({M_i}\) is an algebraic cover of \({M_{i-1}}\) as then by induction
  it follows that \({M_i}\) is an algebraic cover of \({M}\).  Note that
  the requirement \({M_i={(M_i^\eq)}_{\acl^\eq(\emptyset)}}\) can be
  made true as it holds for \({M_0}\) and we can just go over to
  \({{(M_i^\eq)}_{\acl^\eq(\emptyset)}}\) (and preserve that it is an
  algebraic cover) inductively by Lemma~\ref{weigeneralfinitecover}
  and Lemma~\ref{addingaclemptytoalgebraiccover}.

  We construct all finite covers
  \({{(\hat M_i)}_{p,\phi_{i}:1\le i\le m}}\) for some type \({p}\) in
  \({S_{M_i}(\emptyset)}\) and \({\phi_{i}(x,y,z)\in L_i}\) some
  formulae such that there is a \({k_{i}\in \mathbb N}\) such that for
  any \({b,a\in M_i}\) we have
  \({M_i\models\exists^{=k}z\phi_i(b,a,z)}\).  We then amalgamate all of all these
  covers \({{(M_i)}_{p,\phi_{i}:1\le i\le m}}\) together into a new
  structure \({M_{i+1}}\), as in Section~\ref{AmalgSec}. By adding parameters, we may assume that $\acl(\emptyset) = \dcl(\emptyset)$ in $M_{i+1}$. We may also add in imaginary elements.
  
    Let \({T_{i+1}}\) be the theory of
  \({M_{i+1}}\).   By Proposition~\ref{coveroverparameter} we have that for any \({M_i}\) the
  following holds
  (*):\newline
  for any independent sequence
  \({(a_i:1\le i\le n)}\) in \(M_{i}\) (which is from
  \({S_{M_i}(\emptyset)}\)) and any \({d}\) in the set
  \begin{multline*}
    \Bigl(\acl^{M_i} (a_1\ldots a_{n-1})\cap \dcl^{M_i} \bigl(
    \bigcup_{j=1}^{n-1}\acl^{M_i} (a_1\ldots \hat a_j \ldots
    a_n)\bigr)\Bigr) \\-\dcl^{M_i}
    \bigl(\bigcup_{j=1}^{n-1}\acl^{M_i}(a_1\ldots \hat a_j \ldots
    a_{n-1})\bigr),
  \end{multline*}
  we have
  \[d\in \dcl^{M_{i+1}}
    \bigl(\bigcup_{j=1}^{n-1}\acl^{M_{i+1}}(a_1\ldots \hat a_j \ldots
    a_{n-1})\bigr).\] Take \({M^*=\bigcup_{i\in \mathbb N} M_i}\) and
  \({T^*}\) its theory. Then  \({M^*}\) is an algebraic cover of \({M}\).
  It has elimination of imaginaries (in fact \({M^*={(M^*)}^\eq}\)), as
  any imaginary of \({M^*}\) is already an imaginary of some \({M_i}\)
  and hence an element of \({M_{i+1}}\).

  We claim that \({T^*}\) has \({n}\)-amalgamation for every \({n}\).
  For that we check that Property \({B(n)}\) holds for every \({n}\) to finish
  the proof. So take \({d}\) in
  \[
    \Bigl(\acl^{M^*} (a_1\ldots a_{n-1})\cap \dcl^{M^*} \bigl(
    \bigcup_{j=1}^{n-1}\acl^{M^*} (a_1\ldots \hat a_j \ldots
    a_n)\bigr)\Bigr)
  \]
  for \({a_1,\ldots ,a_n}\) some independent sequence in \({M^*}\).
  Now we have that these \({a_1,\ldots ,a_n}\) are part of some
  \({M_j}\).  We can then find some \({M_i}\) (with \({i\ge j}\)) such
  that \({d}\) is in
  \[
    \Bigl(\acl^{M_i} (a_1\ldots a_{n-1})\cap \dcl^{M_i} \bigl(
    \bigcup_{j\neq n}\acl^{M_i} (a_1\ldots \hat a_j \ldots
    a_n)\bigr)\Bigr)
  \]
  but then as already noted we have
  \[d\in \dcl^{M_{i+1}} \bigl(\bigcup_{j=1}^{n-1}\acl^{M_{i+1}}(a_1\ldots \hat
    a_j \ldots a_{n-1})\bigr). \qedhere\]
\end{proof}

\section{Separable forking}\label{subchaptersepindep}

We are going to establish results which will show that, under some conditions,
amalgamation problems over parameters can be translated to
amalgamation problems over \({\emptyset}\). This then shows that, in
this case, total uniqueness over the empty set implies total uniqueness
over all sets.

\begin{definition}\rm 
  We say that a theory \({T}\) has \emph{separable forking}
  if it is stable and
  if in \({T^\eq}\) for all sets
  \({A\subset B}\) there exists some \({C}\) such that
  \({A\indep_\emptyset C}\) and \({\acl(AC)=\acl(B)}\).
\end{definition}

\begin{remark}\rm 
  Note that this definition makes sense outside the stable
  context as we can define separability for any theory with a good notion of
  independence. The next lemma will be  true in any such theory. In fact,
  all results apart from the last corollary in this section are true in the context of rosy
  theories. For this see Section 5.1 of \cite{uea63347}.
\end{remark}

It seems reasonable to ask:

\begin{question}\rm 
  Is the notion of separable forking equivalent to any other model
  theoretic notion?
\end{question}

We will show that almost strongly minimal theories with a $0$-definable strongly minimal set have a separable independence
notion.

\begin{lemma}\label{sec:separ-indep-noti-2}
  If \({T}\) has separable forking, then any algebraic
  cover \({T_1}\) of \({T}\) has separable forking.
\end{lemma}

\begin{proof}  Note first that we may assume that both $T_1$ and $T$ have weak elimination of  imaginaries (by including all imaginary sorts, if necessary). 

Let $M_1 \models T_1$ and let $M$ be the $T$-part of $M_1$. Suppose $A_1 \subseteq B_1 \subseteq M_1$ and let $A = A_1 \cap M$ and $B = B_1 \cap M$. We may assume that $A_1, B_1$ are algebraically closed in $M_1$ and therefore $A, B$ are algebraically closed in $M$. Note that by Lemma~\ref{weigeneralfinitecover}, we have $A_1 = \acl(A)$ (in $M_1$) and $B_1 = \acl(B)$.

By separability in $M$, there is an algebraically closed $C \subseteq B$ with $A \ind^{M}_\emptyset C$ and $\acl(AC) = B$. As $M$ is fully embedded in $M_1$ we have $A\ind^{M_1}_\emptyset C$ and therefore $A_1 \ind^{M_1}_\emptyset C_1$, where $C_1 = \acl(C)$ (in $M_1$). As $B_1 = \acl(A_1C_1)$, the result follows.
\end{proof}

Recall that a theory has \emph{geometric elimination of imaginaries} if every imaginary element is inter-algebraic with a tuple of real elements.

\begin{lemma}\label{sec:separ-indep-noti-4}
 Suppose $T$ is a stable theory with \({U}\)-rank 1
 with geometric elimination of
  imaginaries. Then \({T}\) has separable forking.
\end{lemma}

\begin{proof}
Let $M$ be a model of $T$.  By geometric elimination of
  imaginaries, it is enough to check that forking  is separable in the real
  elements.
  
  Fix \({A,B}\) with $A \subseteq B \subseteq M$.
  We construct an independent sequence  \(\{a_i\mid i<\alpha\}\) with  \({\acl(\{a_i\mid i<\alpha\})=\acl(B)}\) such that for some
  \({\lambda<\alpha}\), we have that \({\acl(\{a_i:i<\lambda)\})=\acl(A)}\). Indeed if we are able to do this, then \({C=\{a_i\mid \lambda\le i<\alpha\}}\) is independent of \({A}\) and \({\acl(AC)=\acl(B)}\) and hence we are finished.

  So take any \({a_{0}\in \acl(A)}\). If \(\{a_{i}\mid i<\beta\}\) is constructed pick any \({a_{\beta}\in \acl(A)-\acl(\{a_{i}\mid i<\beta\})}\) or if  \({\acl(A)-\acl(\{a_{i}\mid i<\beta\})=\emptyset}\) then pick any \({a_{\beta}\in \acl(A)-\acl(\{a_{i}\mid i<\beta\}}\). By
  \({U}\)-rank 1 we then know that this sequence is indeed independent, as \({U  (a_{\beta}/ (a_{i}\mid i<\beta))=1}\) by construction.
\end{proof}

\begin{lemma}\label{sec:separ-indep-noti-3}
Suppose \({T_1}\) is an almost strongly minimal \({L_1}\)-theory  which has a \({0}\)-definable
  strongly minimal formula $\phi$ with geometric elimination of imaginaries. Then  \({T_1}\) has a separable forking.
\end{lemma}

\begin{proof} Let $M_1$ be a saturated model of $T_1$. We can regard $M = \phi(M_1)$ (with the induced structure from $M_1$) as a structure in its own right which is fully embedded in $M_1$. Then $M_1$ is an algebraic cover of $M$. By  Lemma~\ref{sec:separ-indep-noti-4}, $M$ has separable forking. So by Lemma~\ref{sec:separ-indep-noti-2}, the same is true of $M_1$.
\end{proof}

The following result is the main use of the separable forking.
\begin{theorem}\label{sepIndepAmalgoverParameters}
 Suppose $T$ is a stable theory with separable forking.  Further suppose that \({T}\) has
  \({l}\)-uniqueness over \({\emptyset}\) for all \({2\le l\le N+1}\).
  Then \({T}\) has \({k}\)-uniqueness over any set for all \({2\le k \le N}\).
\end{theorem}
\begin{proof} As before, we work in a monster model $\mathbb{M}$ of $T = T^{\eq}$. Let $A$ be algebraically closed. By Proposition~\ref{12} it will suffice to show that property $B(N)$ holds over $A$. So let $a_0,\ldots, a_{N-1}$ be independent over $A$ and, for uniformity of notation, let $a_N$ be an enumeration of $A$. Let  
\[\sigma \in \Aut(\acl(a_0,\ldots, a_{N-2}a_N)/\bigcup_{i = 0}^{N-2} \acl(a_0\ldots \a_i \ldots a_{N-2} a_N)).\]
We need to show that 
\[\sigma \in \Aut(\acl(a_0,\ldots, a_{N-2}a_N)/\bigcup_{i = 0}^{N-2} \acl(a_0\ldots \a_i \ldots a_{N-2} a_{N-1}a_N)).\]

By separability we may assume that $a_i \ind A$  for $i \leq N-1$. By properties of forking, it follows that  $a_0,\ldots, a_{N-1}, a_N$ is an independent sequence (over $\emptyset$). 

First, note that by Corollary~\ref{16} there is an elementary map $\tau$ with  $\tau \vert \acl(a_0,\ldots, a_{N-2}) = \sigma$ and $\tau\vert \acl(a_0,\ldots,\hat{a_i},\ldots, a_{N-2}, a_{N-1})$ equal to the identity map, for $i \leq N-2$. 

A further application of Corollary~\ref{16} then shows that there is an elementary map which is equal to $\sigma$ on $\acl(a_0,\ldots, a_{N-2},a_N)$; is equal to $\tau$ on $\acl(a_0,\ldots, a_{N-1})$; and is the identity on $\acl(a_0,\ldots, \hat{a_i}, \ldots, a_{N-1}, a_{N})$, for all $i \leq N-2$. 
\end{proof}

\begin{corollary}\label{total}
  Let \({T}\) be a stable theory which has separable forking and \({dcl(\emptyset)=acl(\emptyset)}\).  Then
  \({T}\) has an algebraic cover which has
  \({n}\)-amalgamation and \({n}\)-uniqueness for every \({n}\) and over every set.
\end{corollary}

\begin{proof}
  Use Theorem~\ref{uniquenesscor}
  and Theorem~\ref{sepIndepAmalgoverParameters} to conclude that the theory has an
  algebraic cover with \({n}\)-uniqueness over every set.  
\end{proof}

\printbibliography

\end{document}